\documentclass[preprint,12pt]{elsarticle}

\usepackage{amsmath,amssymb,amsthm}
\usepackage{tikz}
\usepackage{url}
\usepackage{hyperref}
\journal{Discrete Mathematics}

\newtheorem{theorem}{Theorem}
\newtheorem{problem}{Problem}
\newtheorem{lemma}{Lemma}
\newtheorem{proposition}{Proposition}
\newtheorem{corollary}{Corollary}
\theoremstyle{definition}
\newtheorem{definition}{Definition}
\theoremstyle{remark}
\newtheorem{remark}{Remark}

\begin{document}

\begin{frontmatter}

\title{On the Packing Coloring Gap of Graphs\tnoteref{t1}}
\tnotetext[t1]{This work is supported by the Scientific and Technological Research Council of Türkiye (TÜBİTAK) under Grant No.~124F114.}

\author{
Batoul Tarhini \quad Didem Gözüpek
\\[1ex]
\texttt{ Emails:\{batoultarhini, didem.gozupek\}@gtu.edu.tr }
}
\affiliation{organization={Gebze Technical University, Department of Computer Engineering},
            addressline={}, 
            city={Gebze},
            postcode={41400}, 
            state={Kocaeli},
            country={Türkiye}}

\begin{abstract}
The packing chromatic number of a graph is the minimum number of colors
for which the graph admits a packing coloring.
This distance-based parameter may change under local
structural modifications of the graph. In this paper, we introduce the \emph{packing coloring gap}, defined as the maximum decrease in the packing chromatic number caused by the deletion of a single vertex. We focus on trees and determine the packing coloring gap for caterpillars.
We further extend these results to caterpillars under the corona operation with $K_1$.
In addition, we present examples of graphs with packing coloring gap
zero, one, and arbitrarily large.
\end{abstract}


\end{frontmatter}

\section{Introduction}

All graphs considered in this paper are finite, simple, and connected.
For a graph $G$, we denote by $V(G)$ and $E(G)$ its vertex and edge sets,
respectively. For vertices $u,v\in V(G)$, the distance $d_G(u,v)$ between $u$ and $v$ is defined
as the length of a shortest $u$--$v$ path in $G$. For a vertex $v\in V(G)$,
the graph obtained by deleting $v$ is denoted by $G-v$.

A \emph{packing coloring} of a graph $G$ is a vertex coloring
$c : V(G) \rightarrow \{1,2,\dots,k\}$ such that for every color $i \in \{1,2,\dots,k\}$, any two distinct vertices $u,v \in V(G)$ with $c(u)=c(v)=i$ satisfy
$d_G(u,v) > i$.
The minimum number $k$ of colors required in a packing coloring of $G$ is
called the \emph{packing chromatic number} of $G$ and is denoted by
$\chi_p(G)$.

The packing chromatic number was originally introduced under the name
\emph{broadcast chromatic number} by Goddard et al.~\cite{Goddard2008}
as a distance-based variant of classical vertex coloring, motivated
in part by frequency assignment problems. Since then, packing colorings have been extensively studied from both
structural and algorithmic perspectives. Numerous results establish bounds for $\chi_p(G)$ on specific graph
classes, including trees, grids, and cubic graphs; see, for example,
\cite{Bresar2011,Klavzar2015}.
On the algorithmic side, determining $\chi_p(G)$ is known to be
computationally difficult, and the problem remains NP-complete even
for restricted classes such as trees.
Since the definition of a packing coloring explicitly depends on
distances in the graph, local structural modifications such as vertex
deletion may significantly affect the packing chromatic number.
Consequently, the behavior of $\chi_p(G)$ under vertex deletion has
received considerable attention;  see, for example,
\cite{Bresar2011,Klavzar2015}.

A graph $G$ is called \emph{packing chromatic vertex-critical}
(or simply \emph{$\chi_p$-critical}) if
$\chi_p(G-v) < \chi_p(G)$ for every vertex $v \in V(G)$.
This notion was systematically studied by Klav\v{z}ar and Rall
in \cite{KlavzarRall2019}, where the authors introduced a refined
framework for measuring vertex criticality in packing colorings. 
Furthermore, they defined the set
\[
\Delta \chi_\rho(G)
= \{ \chi_p(G) - \chi_p(G-x) : x \in V(G) \},
\]
which records all possible decreases of the packing chromatic number
resulting from the deletion of a single vertex.
The notion of vertex-criticality is inherently qualitative:
it asks whether the deletion of \emph{every} vertex strictly reduces
$\chi_p(G)$.
In contrast, the set $\Delta \chi_p(G)$ captures the range of possible
decreases; however, it does not isolate the extremal behavior of vertex deletion.
Motivated by this approach, we introduce a parameter that captures the
\emph{maximum} influence of a single vertex on the packing chromatic
number.
\begin{definition}
For a graph $G$, the \emph{packing coloring gap} of $G$ is defined as
\[
\mu_p(G)
= \max_{v \in V(G)} \bigl( \chi_p(G) - \chi_p(G-v) \bigr).
\]
\end{definition}

Clearly, $\mu_p(G) = \max \Delta \chi_p(G)$, and hence this parameter
naturally extends the framework introduced in
\cite{KlavzarRall2019}.
While $\chi_p$-critical graphs correspond to those for which
$\Delta \chi_p(G) \subseteq \mathbb{N}^+$, the value of $\mu_p(G)$
provides a quantitative measure of vertex criticality.
Large values of $\mu_p(G)$ indicate the presence of vertices whose
removal causes a substantial relaxation of the distance constraints
imposed by a packing coloring.

\medskip 
Studying packing chromatic vertex-criticality and the packing coloring gap is important because it highlights vertices that have a significant influence on distance constraints in the graph. Understanding which vertices cause significant decreases in $\chi_p(G)$
 can inform both theoretical and practical applications, such as network design, resource allocation, and frequency assignment, where maintaining distance-based constraints is crucial.

\medskip
The variability of the packing coloring gap motivates the study of graphs for which $\mu_p(G)$ is small or bounded. 
In this paper, we first present explicit examples of graphs with packing coloring gaps $0$ or $1$, 
as well as examples of graphs with unbounded gaps. 
We then focus on graphs with small or bounded packing coloring gaps. 
To this end, we begin with trees, starting with specific classes such as caterpillars. We show that, for every caterpillar $G$, the
packing coloring gap satisfies $\mu_p(G)\le 2$. Caterpillars constitute a natural starting point, since their packing
chromatic number has been studied in \cite{Sloper2004} and is known to be bounded by $7$. Moreover, the authors in \cite{KlavzarRall2019}, considered the class of caterpillars and
proved that for every $k \in \{1,\dots,7\}$, there exists a $\chi_p$-critical
caterpillar with packing chromatic number $k$. In the same spirit, we show that for every $k \in \{1,\dots,7\}$, there exists a caterpillar $T$
such that $\chi_p(T)=k$ and $\mu_p(T)=1$.
We then extend our investigation to more general tree structures, such as
lobsters, focusing in particular on caterpillars under the corona
operation with $K_1$. For this class, we first establish an upper bound on the packing
chromatic number (at most 7) and then show that the packing coloring gap remains
bounded by 2.
Our results provide insights into how local graph structure influences the packing coloring gap.

\section{Preliminaries}

We present several basic observations on the packing coloring gap.
\begin{remark}
For every graph $G$, we have $\mu_p(G)\ge 0$. Moreover, $\mu_p(G)=0$ if and only if
\[
\chi_p(G-v)=\chi_p(G) \quad \text{for all } v\in V(G).
\]
\end{remark}

\begin{lemma}\label{lem:basic-gap}
For every graph $G$ and every vertex $v\in V(G)$,
\[
\chi_p(G-v) \ge \chi_p(G)-\mu_p(G).
\]
\end{lemma}

\begin{proof}
The statement follows directly from the definition of $\mu_p(G)$.
\end{proof}

\begin{lemma}\label{lem:gap-upper}
For every graph $G$ with at least two vertices, we have
\[
\mu_p(G) \le \chi_p(G)-1.
\]
\end{lemma}

\begin{proof}
Let $v\in V(G)$. Since $G$ has at least two vertices, the graph $G-v$  has at least one vertex and thus
$\chi_p(G-v)\ge 1$. Hence,
\[
\chi_p(G)-\chi_p(G-v) \le \chi_p(G)-1.
\]
Taking the maximum over all vertices yields the result.
\begin{remark}
Note that the assumption that $G$ has at least two vertices is necessary:
if $G$ consists of a single vertex, then $\chi_p(G)=1$ and deleting 
the vertex yields the graph with no vertices, for which $\chi_p(G-v)=0$. In this 
case, the packing coloring gap is $\mu_p(G)=1$, which is not less 
than $\chi_p(G)-1=0$, so the lemma would not hold.
\end{remark}
\end{proof}

We will also use the following characterization of graphs
with packing chromatic number two, proved by Bre\v{s}ar et al.~\cite{Bresar2007}.

\begin{proposition} ~\cite{Bresar2007} \label{prop:chi2}
A connected graph $G$ satisfies $\chi_p(G)=2$
if and only if $G$ is a star.
\end{proposition}

\section{Graphs with Packing Coloring Gap At Most One}


In this section, we give examples of graphs $G$ satisfying $\mu_p(G)\leq 1$.
\subsection{Complete Graphs}
For the complete graph $K_n$ with $n \ge 2$, we have $\chi_p(K_n)=n$, since every pair of vertices is at distance~$1$.
Deleting any vertex yields $K_{n-1}$, whose packing chromatic number is $n-1$.
Therefore,
\[
\mu_p(K_n)=\chi_p(K_n)-\chi_p(K_n-v)=1.
\]
\subsection{Complete Bipartite Graphs}

Let $K_{m,n}$ be the complete bipartite graph with partition sizes $m$ and $n$, where $m \le n$. It is shown in \cite{UrendaEtAl2022} that
\[
\chi_p(K_{m,n}) = 1 + \min(m,n).
\] Deleting any vertex may leave the packing chromatic number unchanged or decrease it by at most one, so
\[
\mu_p(K_{m,n}) \le 1.
\]

Thus, complete bipartite graphs provide natural examples of graphs with packing coloring gap at most $1$.

\subsection{Paths and cycles}

The packing chromatic numbers of paths and cycles are well known (see ~\cite{Goddard2008}).
For every $n \ge 1$, we have

\[
\chi_p(P_n)=
\begin{cases}
1, & n=1,\\
2, & n=2,3,\\
3, & n \ge 4,
\end{cases}
\]

and

\[
\chi_p(C_n)=
\begin{cases}
3,  \text{if } n=3 \text{ or } n \equiv 0 \pmod{4},\\
4,  \text{otherwise}
\end{cases}
\]

These values follow from explicit packing colorings.
For instance, the infinite path admits a repeating $(1,2,1,3)$-pattern,
which shows that $\chi_p(P_n)\le 3$ for $n\ge 4$, and it is easy to
verify that two colors do not suffice. Similar periodic constructions
yield the values for cycles.

We now determine the packing coloring gap for these graphs.

\begin{proposition}
For every $n\ge 1$,
\[
\mu_p(P_n)\le 1
\quad\text{and}\quad
\mu_p(C_n)\le 1.
\]
\end{proposition}

\begin{proof}
Deleting a vertex from $P_n$ produces $P_{n-1}$.
From the explicit formula above, the packing chromatic number
changes only in small exceptional cases
$n=1,2,3,4$, and in each case the decrease is at most one.
Hence $\mu_p(P_n)\le 1$.

Similarly, deleting a vertex from $C_n$ produces $P_{n-1}$.
Comparing the known values of $\chi_p(C_n)$ and $\chi_p(P_{n-1})$
shows that the difference is never greater than one.
Therefore $\mu_p(C_n)\le 1$.
\end{proof}

\subsection{Spiders}

A \emph{spider} is a tree consisting of a single central vertex, called the \emph{body}, 
from which several disjoint paths, called \emph{legs}, originate. Formally, a spider 
is a tree in which exactly one vertex has degree greater than 2 (the body), 
and all other vertices lie along paths starting from this central vertex.

We show in the following that the packing coloring gap is at most 1 for spiders.
\begin{proposition}
Let $S$ be a spider. Then $\mu(S)=1$ in the two cases described below, and $\mu(S)=0$ otherwise.
\end{proposition}
\begin{proof}
    Let $S$ be a spider with central vertex $c$. We have $\chi_p(S)\leq 3$. Indeed, one can color $c$ with color $3$ and color each leg by repeating the sequence $1,2,1,3$, which satisfies the packing constraints.\\
We now examine the effect of deleting a vertex from $S$. The value $\mu(S)=1$ occurs in the following two cases.
\begin{itemize}
    \item Case 1: If all legs have order at most $4$. In this case, removing the central vertex produces a disjoint union of paths, each of order at most $3$. Consequently, the packing chromatic number decreases from $3$ to $2$, or from $2$ to $1$.
    
    \item Case 2: If $S$ is a star  with exactly one leg whose order is between $3$ and $6$. In this case, removing the third vertex of this leg yields a disjoint union of a star and possibly a path of order at most $3$. Hence, the packing chromatic number decreases from $3$ to $2$.
\end{itemize}
Conversely, if neither of the above conditions holds, then removing any vertex still yields a graph that contains a path of order at least $4$, which has packing chromatic number $3$. Therefore, $\mu(S)=0$.
\end{proof}

\section{Graphs with Bounded Packing Coloring Gap}

In this section, we study graphs whose packing coloring gap is bounded by a fixed constant. The parameter \(\mu_p(G)\) measures the maximum decrease of the packing chromatic number under deletion of a single vertex. While in general graphs this parameter can be arbitrarily large (see section \ref{unboundedSection}), it is natural to investigate how \(\mu_p(G)\) behaves for specific graph classes and how structural properties of a graph restrict its magnitude.

We begin our study with trees due to their simple structure and well-understood distance relationships, making them fundamental building blocks for analyzing packing coloring behavior. Moreover, our approach is motivated by the fact that determining the packing chromatic
number is NP-complete even for trees \cite{{Argiroffo2014}}. Thus, although trees have a simple
structure, the problem remains computationally difficult. This makes it
natural to study related parameters, such as the packing coloring gap,
within this class of graphs.  
We focus first on \emph{caterpillars} which are trees in
which all vertices of degree at least two lie on a central path. The
structure of caterpillars allows for an explicit analysis of the packing
coloring gap.

After establishing results for caterpillars, we extend our analysis to more complex tree structures, including generalized caterpillars and other constructions obtained through standard graph operations. A notable example is the \emph{corona operation} with \(K_1\), which attaches a pendant vertex to every vertex of a graph. These constructions allow us to explore how \(\mu_p(G)\) behaves under controlled graph extensions and provide insight into more general classes of graphs with bounded packing coloring gaps.

\subsection{Packing Coloring Gap of Caterpillars}

Let $T$ be a caterpillar with a central path $P$. We denote the vertices
of $P$ by $v_1, v_2, \dots, v_\ell$ in order. The leaves attached to
each $v_i$ are denoted by $L(v_i)$ for $i \in [l]$. The structure of $T$ allows us to
analyze packing colorings iteratively along the path, taking into account
the contribution of leaves to the distance constraints.

Sloper \cite{Sloper2004} proved that the packing chromatic number of any caterpillar is at
most $7$.
Hence, for any caterpillar $T$, we have $\chi_p(T)\in\{1,2,\dots,7\}$.
We prove Theorem~\ref{thm:caterpillar-gap} by distinguishing cases
according to the value of $\chi_p(T)$.\\ Moreover, we will show that the bound of the packing coloring gap is sharp.
\begin{theorem}\label{thm:caterpillar-gap}
Let $T$ be a caterpillar. Then
\[
\mu_p(T) \le 2.
\]
\end{theorem} 
\begin{proof}
First, we distinguish between the deletion of a leaf and the deletion of a vertex
on the central path of the caterpillar. 

If $v$ is a leaf, then $T-v$ is simply a caterpillar obtained by removing
a pendant vertex, and the underlying central path remains unchanged. 

In contrast, let $v'$ be the neighbor of $v$ on the central path, then removing $v'$
breaks the spine of the caterpillar and results in two disjoint caterpillars,
together with an isolated vertex, $v$. In particular, these resulting
caterpillars are subgraphs of $T-v$.
Consequently,
\[
\chi_p(T-v) \ge \chi_p(T-v'),
\]
Hence, the deletion of a vertex on the central path may produce
a larger decrease in the packing chromatic number.
For this reason, in the analysis of the packing coloring gap, it suffices
to consider vertices lying on the central path.

\medskip
\noindent\textbf{Case 1: $\chi_p(T)\in\{1,2,3\}$.}\\
In both cases where $\chi_p(T)=1$ and $2$, it is trivial that the gap is at most 1.\\
If $\chi_p(T)=3$, then $T-v$ is nonempty for every $v\in V(T)$, and hence
$\chi_p(T-v)\ge 2$. Therefore,
$\chi_p(T)-\chi_p(T-v)\le 1$, and so $\mu_p(T)\le 1$.

\medskip
\noindent\textbf{Case 2: $\chi_p(T)=4$.}
We prove that $\mu_p(T)=1$, except in a specific case where
$\mu_p(T)=2$. Let $v\in V(T)$. Clearly, $\chi_p(T-v)\neq 1$, because otherwise $T$ is a star and has $\chi_p(T)=2$ by Proposition~\ref{prop:chi2}. Likewise, we have $\chi_p(T-v)=2$ if and only if $T-v$ is a star.
Hence, the removal of the vertex $v$ results in a disjoint union of stars. Therefore, the gap equals $2$ exactly when $\chi_p(T-v)=2$, i.e., when $T-v$ is a disjoint union of stars. This occurs only when the central path has order $3$.

Otherwise, the graph configuration shown in Figure~\ref{fig:case2} appears as a subgraph of the graph obtained after deleting any vertex, and thus deleting a vertex does not yield a disjoint union of stars. Hence, for all other caterpillars with \(\chi_p(T)=4\), we have
$\mu_p(T)=1$.

\begin{figure}[h]
\centering
\begin{tikzpicture}[scale=1,
  spine/.style={circle, draw, fill=black, inner sep=2pt},
  removed/.style={circle, draw, thick, inner sep=6pt},
  leaf/.style={circle, draw, fill=black, inner sep=1.5pt}
]

\node[spine] (v1) at (0,0) {};
\node[spine] (v2) at (2,0) {};

\draw (v1)--(v2);

\node[leaf] at (-0.4,1.1) {};
\node at (0,1.1) {$\cdots$};
\node[leaf] at (0.4,1.1) {};
\draw (v1)--(-0.4,1.1);
\draw (v1)--(0.4,1.1);

\node[leaf] at (1.6,1.1) {};
\node at (2,1.1) {$\cdots$};
\node[leaf] at (2.4,1.1) {};
\draw (v2)--(1.6,1.1);
\draw (v2)--(2.4,1.1);

\end{tikzpicture}
\caption{A configuration that appears as a subgraph of $T-v$ for any vertex $v$. Each vertex shown is incident with at least one leaf, possibly together with additional leaves.}
\label{fig:case2}
\end{figure}
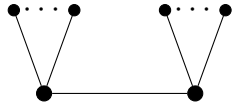

\noindent\textbf{Case 3: $\chi_p(T)=5$.}  
We show that, in this case, the packing coloring gap never exceeds $2$; that is, for every vertex $v \in V(T)$, 
\[
\chi_p(T-v) \ge \chi_p(T) - 2.
\]

Suppose for contradiction that there exists a vertex $v$ such that
\[
\chi_p(T)-\chi_p(T-v)=3.
\]
Then $\chi_p(T-v) = 2$, which implies that $T-v$ is a disjoint union of stars. Re-adding the vertex $v$ back to this disjoint union of stars produces the tree $T$. However, such a graph, obtained by linking two stars through a single vertex(possibly with its own leaves), has a packing chromatic number at most $4$, which contradicts our assumption that $\chi_p(T)=5$.
Therefore, in this case, the packing coloring gap satisfies 
\[
\mu_p(T) \leq 2.
\]

\medskip
\noindent\textbf{Case 4: $\chi_p(T)=6$.}
Suppose, for contradiction, that $\mu_p(T)=3$.
Then there exists a vertex $v\in V(T)$ such that $\chi_p(T-v)=3$.\\
Hence, the graph $T-v$ is a disjoint union of subgraphs
$ T_1$ and $ T_2$,
each of which admits a packing coloring using the color set
$\{1,2,3\}$. We now consider feasible packing colorings of $ T_1$ and $ T_2$, with at most 3 colors, reinsert the vertex $v$ and analyze the possible conflicts
that may arise. Observe that the number of vertices involved in conflicts is even. Indeed, any conflict must occur between vertices belonging to $T_1$ and $T_2$, since each of $T_1$ and $T_2$ is already packing colored. In particular, if three vertices were involved in a conflict, then two of them would already be in conflict within $T_1$ or $T_2$, which is impossible.
Now let $v_{i-1}$ and $v_{i+1}$ denote the vertices immediately preceding
and succeeding $v$ on the central path of the caterpillar $T$.

\medskip
\noindent\emph{First conflict.}
Assume that both $v_{i-1}$ and $v_{i+1}$ receive color~$2$.
In this case, recolor $v_{i-1}$ with color~$4$ and assign color~$5$
to the vertex $v$.
No further conflicts arise, and thus at most two new colors are needed.

\medskip
\noindent\emph{Second conflict.}
Assume that $v_{i-1}$ and $v_{i+1}$ both receive color~$3$.
Again, recolor $v_{i-1}$ with color~$4$ and assign color~$5$
to the vertex $v$.
This resolves all conflicts, using at most two additional colors.

\medskip
\noindent\emph{Third conflict.}
Assume that $v_{i-1}$ and $v_{i+2}$ both receive color~$3$.
Again, recolor $v_{i-1}$ with color~$4$ and assign color~$5$
to the vertex $v$.\\ Similarly, if $v_{i-2}$ and $v_{i+1}$ both receive color~$3$, we proceed in the same way by recoloring $v_{i-2}$ with color~$4$ and assigning color~$5$ to the vertex $v$.\\
This resolves all conflicts, using at most two additional colors

\medskip
Observe that non of these two conflicts can occur simultaneously. Therefore, in all cases, the reinsertion of $v$ requires at most two
additional colors.
Consequently, we obtain
\[
\chi_p(T) \le 5,
\]
which contradicts the assumption that $\chi_p(T)=6$.

Hence, the packing coloring gap cannot be equal to $3$ in this case,
and we conclude that
\[
\mu_p(T) \le 2.
\]
\medskip
\noindent\textbf{Case 5: $\chi_p(T)=7$.}
Suppose, for contradiction, that $\mu_p(T)=3$.
Then there exists a vertex $v\in V(T)$ such that $\chi_p(T-v)=4.$ Hence, the graph $T-v$ can be expressed as a disjoint union
\[
T-v = T_1 \cup T_2,
\]
where each component admits a packing coloring using the color set
$\{1,2,3,4\}$.

We now reinsert the vertex $v$ and analyze the possible conflicts that
may arise with respect to the packing coloring of $T-v$.
Let $v_{i-1}$ and $v_{i+1}$ denote the neighbors of $v$ on the central
path of the caterpillar $T$.

\medskip

If exactly two vertices are involved in a conflict by the reinsertion of $v$, then the conflict 
can be resolved by recoloring one of the conflicting vertices with a new
color, say 5, and assigning a new color to $v$, say 6. Thus, at most two additional colors are required.

Consequently, we obtain
\[
\chi_p(T) \le 6,
\]
which contradicts the assumption that $\chi_p(T)=7$.

Assume now that four vertices are involved in conflicts that
occur in pairs.
Consider the vertices on the central path in order
\[
v_{i-2},\, v_{i-1},\, v,\, v_{i+1},\, v_{i+2}.
\]

\medskip
\noindent\emph{Conflict 1.}  
Suppose that the colors of the vertices
$v_{i-2},\, v_{i-1},\, v,\, v_{i+1},\, v_{i+2}$
form the sequence $4,\,2,\,v,\,2,\,4.$

\medskip
\noindent\emph{Subcase 1.}  
Assume that both $v_{i-3}$ and $v_{i+3}$ receive color~$3$.  
We recolor the vertices from $v_{i-3}$ to $v_{i+3}$ according to the sequence 
$5\,\,3\,\,2\,\,4\,\,6\,\,3\,\,5$. This recoloring is valid since $v_{i-3}$ and $v_{i+3}$ are at distance~$6$ and can both receive color~$5$.  
The vertices $v_{i-2}$ and $v_{i+2}$ can be assigned color~$3$,  
$v_{i+1}$ is recolored with color~$6$, and $v$ is assigned color~$4$.

\medskip
\noindent\emph{Subcase 2.}  
Assume that both $v_{i-3}$ and $v_{i+3}$ receive color~$1$. We recolor the vertices from $v_{i-3}$ to $v_{i+3}$ according to the sequence
$1\,\,4\,\,2\,\,3\,\,5\,\,6\,\,1.$ This recoloring is valid because $v$ is colored~$3$, which is sufficiently far from any other vertex colored~$3$, and $v_{i+1}$ and $v_{i+2}$ receive new colors to resolve their conflicts.

\medskip
\noindent\emph{Subcase 3.}  
Assume $v_{i-3}$ has color~$3$ and $v_{i+3}$ has color~$1$.  
We recolor the vertices from $v$ to $v_{i+4}$ according to the sequence $6\,\,3\,\,2\,\,1\,\,5.$ This recoloring is valid since the new color of $v_{i+4}$ allows $v_{i+2}$
to take color~$2$ and $v_{i+1}$ to take color~$3$, ensuring that there are no conflicts with other vertices on the path.

\medskip
\noindent\emph{Remaining conflicts.}  
The remaining conflict patterns are as follows:
\[
4\,3\,\,v\,\,3\,4 \,,\quad 3\,4\,\,v\,\,3\,4 \,,\quad 4\,3\,\,v\,\,4\,3
\]

All of these cases are treated the same way; that is, recoloring the vertices $v_{i+1}$ and $v_{i+2}$ with colors~$5$ and~$6$,  
respectively, and assigning color~$2$ to the vertex $v$.  

This recoloring resolves all conflicts, and the packing coloring remains valid.\\ \\
Combining all cases $\chi_p(T) = 1, \dots, 7$ and analyzing the possible
conflicts that arise when a vertex $v$ is reinserted, we see that in neither of these cases the packing coloring gap can reach $3$.\\
In each case, any potential conflict can be resolved by introducing at most two additional colors, and the resulting coloring remains valid.  

Hence, we conclude that for any caterpillar $T$, $\mu_p(T) \le 2.$\\
This completes the proof of the theorem.
\end{proof}
As stated in the introduction, the authors in \cite{KlavzarRall2019} investigated the class of caterpillars and established that for every $k \in \{1,\ldots,7\}$ there exists a $\chi_p$-critical caterpillar whose packing chromatic number is $k$. Following this line of research, we prove the following result, as shown in Theorem~\ref{gap1caterpillar} and the accompanying figures.
\begin{theorem} \label{gap1caterpillar}
For every $k \in \{1, \dots , 7\}$, there exists a caterpillar $T$ such that 
its packing chromatic number $\chi_p(T) = k$ and its packing coloring gap $\mu_p(T) = 1$.
\end{theorem}

\begin{proof}
For each value of $k$, we provide an explicit example of a caterpillar $T$ 
having a gap $1$, illustrated in a figure. 
In each case, the removal of any red-colored vertex will decrease the packing chromatic number by $1$.
We omit the cases $k = 1$ and $k = 2$, and start with $k = 3$.

\usetikzlibrary{positioning}

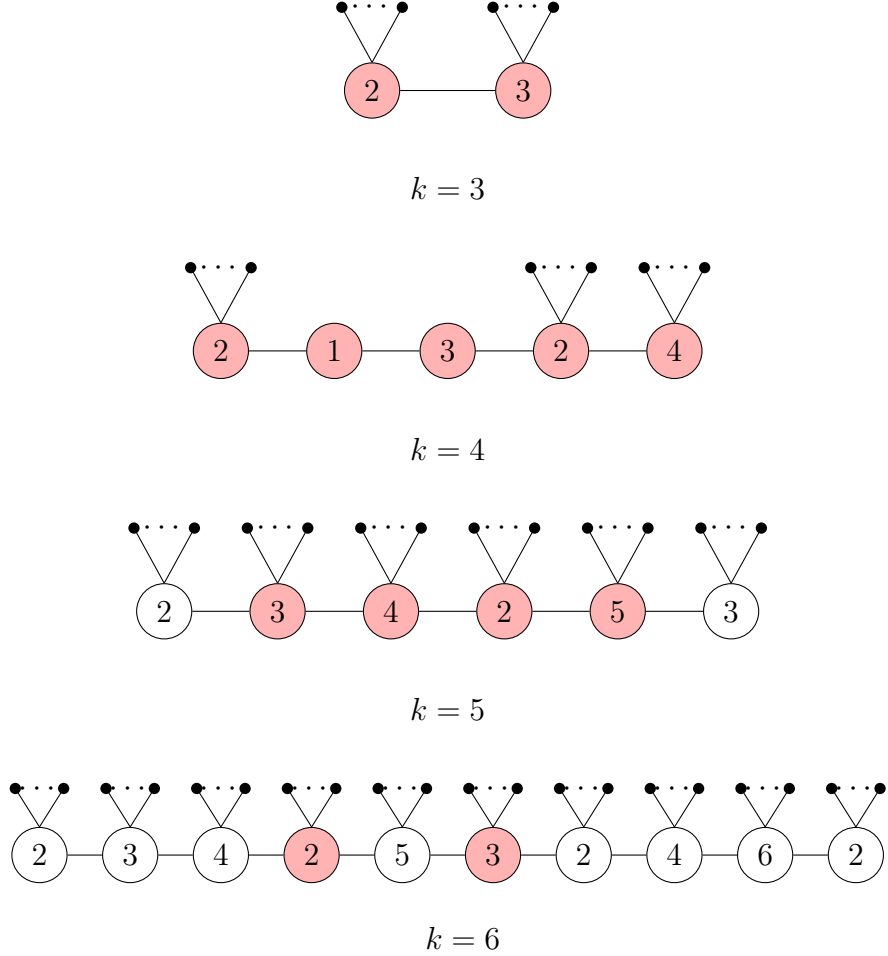
\begin{figure}[h!]
\centering

\tikzset{
    leaf/.style={circle, fill=black, inner sep=1.5pt}
}

\begin{tikzpicture}[scale=1]
\node[draw, circle, minimum size=7mm, fill=red!30] (v1) at (0,0) {2}; 
\node[draw, circle, minimum size=7mm, fill=red!30] (v2) at (2,0) {3}; 

\node[leaf] at (-0.4,1.1) {};
\node at (0,1.1) {$\cdots$};
\node[leaf] at (0.4,1.1) {};
\draw (v1.north) -- (-0.4,1.1);
\draw (v1.north) -- (0.4,1.1);

\node[leaf] at (1.6,1.1) {};
\node at (2,1.1) {$\cdots$};
\node[leaf] at (2.4,1.1) {};
\draw (v2.north) -- (1.6,1.1);
\draw (v2.north) -- (2.4,1.1);

\draw (v1) -- (v2);

\node[below] at (1,-1) {$k=3$};
\end{tikzpicture}

\vspace{0.5cm}

\begin{tikzpicture}[scale=1]
\node[draw, circle, minimum size=7mm, fill=red!30] (v1) at (0,0) {2};
\node[draw, circle, minimum size=7mm, fill=red!30] (v2) at (1.5,0) {1};
\node[draw, circle, minimum size=7mm, fill=red!30] (v3) at (3,0) {3};
\node[draw, circle, minimum size=7mm, fill=red!30] (v4) at (4.5,0) {2};
\node[draw, circle, minimum size=7mm, fill=red!30] (v5) at (6,0) {4};

\node[leaf] at (-0.4,1.1) {};
\node at (0,1.1) {$\cdots$};
\node[leaf] at (0.4,1.1) {};
\draw (v1.north) -- (-0.4,1.1);
\draw (v1.north) -- (0.4,1.1);

\node[leaf] at (4.1,1.1) {};
\node at (4.5,1.1) {$\cdots$};
\node[leaf] at (4.9,1.1) {};
\draw (v4.north) -- (4.1,1.1);
\draw (v4.north) -- (4.9,1.1);

\node[leaf] at (5.6,1.1) {};
\node at (6,1.1) {$\cdots$};
\node[leaf] at (6.4,1.1) {};
\draw (v5.north) -- (5.6,1.1);
\draw (v5.north) -- (6.4,1.1);

\draw (v1) -- (v2) -- (v3) -- (v4) -- (v5);

\node[below] at (3,-1) {$k=4$};
\end{tikzpicture}

\vspace{0.5cm}

\begin{tikzpicture}[scale=1]
\node[draw, circle, minimum size=7mm] (v1) at (0,0) {2};
\node[draw, circle, minimum size=7mm, fill=red!30] (v2) at (1.5,0) {3};
\node[draw, circle, minimum size=7mm, fill=red!30] (v3) at (3,0) {4}; 
\node[draw, circle, minimum size=7mm, fill=red!30] (v4) at (4.5,0) {2}; 
\node[draw, circle, minimum size=7mm, fill=red!30] (v5) at (6,0) {5};
\node[draw, circle, minimum size=7mm] (v6) at (7.5,0) {3};

\foreach \i/\x in {1/0,2/1.5,3/3,4/4.5,5/6,6/7.5} {
    \node[leaf] at (\x-0.4,1.1) {};
    \node at (\x,1.1) {$\cdots$};
    \node[leaf] at (\x+0.4,1.1) {};
    \draw (v\i.north) -- (\x-0.4,1.1);
    \draw (v\i.north) -- (\x+0.4,1.1);
}

\foreach \i/\j in {1/2,2/3,3/4,4/5,5/6} {
    \draw (v\i) -- (v\j);
}

\node[below] at (3.75,-1) {$k=5$};
\end{tikzpicture}

\vspace{0.5cm}

\begin{tikzpicture}[scale=0.8]
\node[draw, circle, minimum size=7mm] (v1) at (0,0) {2};
\node[draw, circle, minimum size=7mm] (v2) at (1.5,0) {3};
\node[draw, circle, minimum size=7mm] (v3) at (3,0) {4};
\node[draw, circle, minimum size=7mm, fill=red!30] (v4) at (4.5,0) {2}; 
\node[draw, circle, minimum size=7mm] (v5) at (6,0) {5};
\node[draw, circle, minimum size=7mm, fill=red!30] (v6) at (7.5,0) {3}; 
\node[draw, circle, minimum size=7mm] (v7) at (9,0) {2};
\node[draw, circle, minimum size=7mm] (v8) at (10.5,0) {4};
\node[draw, circle, minimum size=7mm] (v9) at (12,0) {6};
\node[draw, circle, minimum size=7mm] (v10) at (13.5,0) {2};

\foreach \i/\x in {1/0,2/1.5,3/3,4/4.5,5/6,6/7.5,7/9,8/10.5,9/12,10/13.5} {
    \node[leaf] at (\x-0.4,1.1) {};
    \node at (\x,1.1) {$\cdots$};
    \node[leaf] at (\x+0.4,1.1) {};
    \draw (v\i.north) -- (\x-0.4,1.1);
    \draw (v\i.north) -- (\x+0.4,1.1);
}

\foreach \i/\j in {1/2,2/3,3/4,4/5,5/6,6/7,7/8,8/9,9/10} {
    \draw (v\i) -- (v\j);
}

\node[below] at (7,-1) {$k=6$};
\end{tikzpicture}

\caption{Caterpillars $T$ with $\chi_p(T) = k$ and $\mu_p(T) = 1$ for $k=3,4,5,6$. 
Numbers in spine vertices shows a packing coloring for $T$.}
\label{fig:caterpillars}
\end{figure}
\paragraph{}
For $k=7$, we explain the example as follows. Klavžar et al. showed in ~\cite{KlavzarRall2019} that if $n \ge 35$, then attaching to every vertex of $P_n$ precisely (or more) $6$ leaves yields a caterpillar $T$ with $\chi_p(T) = 7$. 

To construct our example, let $T$ be the caterpillar with $n = 35$ and attach to each vertex of $P_n$ exactly $6$ leaves. Then $\chi_p(T) = 7$. Moreover, removing any vertex on the central path results in a subgraph having one of its components being a caterpillar $T'$ whose central path has order $n'$ with $17 \le n' \le 34$, and $6$ leaves attached to each vertex. Therefore,  $\chi_\rho(T') = 6$. Hence, the packing chromatic gap in this example is $1$.

\end{proof}

\section{Generalized Caterpillars and Corona Products}

Motivated by the results obtained for caterpillars, we extend our study to a natural and structured
generalization obtained by the corona operation.
Recall that for a graph $G$, the corona product $G \odot K_1$ is the graph obtained by attaching a single
pendant vertex to each vertex of $G$.

In this section, we focus on graphs of the form $T \odot K_1$, where $T$ is a caterpillar.
These graphs can be viewed as caterpillars in which every vertex receives an additional pendant neighbor. Our goal is to investigate the impact of the corona operation on the packing coloring gap.
In particular, we prove that for caterpillars, the corona operation preserves the same
upper bound on the packing coloring gap. More precisely, we show that for any caterpillar $T$, we have $\mu_p(T \odot K_1) \le 2$

As in the previous section, and as explained in the proof of Theorem~\ref{thm:caterpillar-gap}, 
we restrict our attention to vertices on the central path, since their deletion 
has the most significant impact on the packing chromatic number, and hence on the gap.

The proof proceeds by analyzing the possible values of the packing chromatic number of
$T \odot K_1$ and showing that, in all cases, the deletion of a vertex cannot decrease the packing
chromatic number by more than two.
\subsection{Packing Chromatic Number of the Corona Product of a Caterpillar with  $ K_1$}

Before studying the packing coloring gap of caterpillars with the corona operation,
we first investigate the packing chromatic number of these graphs.
Indeed, understanding how to efficiently color $T \odot K_1$ is essential in order
to control the effect of vertex deletion and, consequently, the value of the gap.

As mentioned in the discussion preceding Theorem~\ref{thm:caterpillar-gap}, optimal packing colorings for caterpillars typically rely on coloring
the central path using colors from the set $\{2,3,\ldots,7\}$.
However, after applying the corona operation, this strategy can no longer be applied
directly, since each vertex of the caterpillar receives an additional pendant neighbor.
This motivates the introduction of a more refined coloring strategy. A natural first approach would be to color the new pendant vertices with the color $2$.
However, this would prevent the use of color $2$ on the central path, and a similar
obstruction arises for color $3$.
To overcome this difficulty, we employ a coloring strategy that allows the reuse of
colors $2$ and $3$ both on the central path and pendant vertices, while respecting the distance
constraints.  We formalize this in the following proposition.
\begin{proposition}\label{prop:corona-caterpillar}
Let $T$ be a caterpillar. Then
\[
\chi_p(T \odot K_1) \le 7.
\]
\end{proposition}
\begin{proof}
    Let $P=v_1v_2\cdots v_m$ denote the central path of the caterpillar $T$.
For each vertex $v \in V(P)$, we define:
\begin{itemize}
    \item $L_1(v)$ as the set of vertices at distance one from $v$ that do not belong to the central path.
    \item $L_2(v)$ as the set of vertices at distance two from $v$ that do not belong to the central path.
\end{itemize}
We further define
\[
L_1 = \bigcup_{v \in V(P)} L_1(v)
\quad \text{and} \quad
L_2 = \bigcup_{v \in V(P)} L_2(v).
\]

We use the colors $2$ and $3$ alternately on the central path; we never use them on consecutive vertices.
More precisely, if the central path is $v_1v_2\cdots v_m$, then we assign
\[
c(v_1)=2,\quad c(v_3)=3,\quad c(v_5)=2,\quad c(v_7)=3,\ \text{and so on}.
\]
All vertices in $L_2$ are colored with color $2$, except for those belonging to
$L_2(v_i)$ where $v_i$ receives color $2$ on the central path; in that case,
the vertices of $L_2(v_i)$ are colored with color $3$.
Since colors $2$ and $3$ are not used on consecutive vertices of the path, this ensures
that all distance constraints are satisfied.

The coloring of the central path is then completed by applying the following
repeating pattern along the entire path:
\[
(2\,4\,3\,5\,2\,6\,3\,4\,2\,5\,3\,7)^\ast.
\]
This completes the proof.
\end{proof}
\subsection{Packing Coloring Gap of the Corona Product of a Caterpillar with $K_1$}
Based on the bound established in the previous subsection, we are now prepared to study the packing coloring gap of
caterpillars with the corona operation.
\medskip
Before proving our result, we first establish the following
corollary, which will be useful in our analysis of such graphs.

\begin{corollary}\label{cor:chi3-structure}
Let $T$ be a caterpillar. Then
\[
\chi_p(T \odot K_1) = 3
\quad \text{if and only if} \quad
T \text{ is a star.}
\]
\end{corollary}

\begin{proof}
First, observe that the corona of a star graph with $K_1$ clearly has
packing chromatic number $3$.  

Now, let $G = T \odot K_1$ be a caterpillar corona with packing chromatic number
$\chi_p(G) = 3$. The underlying caterpillar $T$ must have packing chromatic number
$2$ or $3$, since it is trivially that $\chi_p(T) \le \chi_p(T \odot K_1)$.\\
Caterpillars with packing chromatic number $2$ are exactly the stars.
Caterpillars with packing chromatic number $3$ have been completely
characterized in~\cite{{Furmanczyk2025}}.
Upon checking all families of such caterpillars under the corona operation, we observe that none of them produces a graph with a packing chromatic number $3$. 
Indeed, in $T \odot K_1$, every vertex of $T$ that originally had zero leaves becomes adjacent to a new pendant vertex.
As a consequence, the caterpillar subgraph of $T \odot K_1$ no longer belongs to any of the graph families characterized in ~\cite{{Furmanczyk2025}}, forcing the packing chromatic number to exceed $3$. \\
Therefore, the only possibility is that $T$ is a star.
\end{proof}
We are now ready to prove our result regarding the gap of the corona product of a caterpillar with $K_1$. We prove the following:  \begin{theorem} \label{thm:corona-gap}
Let $T$ be a caterpillar.
Then the packing coloring gap of the corona graph $T \odot K_1$ satisfies
\[
\mu_p(T \odot K_1) \le 2.
\]  
\end{theorem}

\begin{proof}
The proof proceeds by analyzing the possible values of the packing chromatic number
of $G = T \odot K_1$ and showing that, in each case, the deletion of a vertex cannot
decrease the packing chromatic number by more than two.
\medskip
\noindent\textbf{Case 1: $\chi_p(G) \le 4$.}
In this case, the packing coloring gap satisfies $\mu_p(G) \le 2$, since deleting a vertex cannot decrease the packing chromatic number by more than two because otherwise the resulting graph would have the packing chromatic number $1$, which is possible only for a disjoint union of isolated vertices.

\medskip
\noindent\textbf{Case 2: $\chi_p(G) = 5$.}
Suppose, for contradiction, that there exists a vertex $v \in V(G)$ such that
$\chi_p(G-v) = 2$.
Then $G-v$ is a disjoint union of stars.
However, this implies that either $G$ is not of the form $T \odot K_1$ for any
caterpillar $T$, or that $\chi_p(G) \le 4$, both of which contradict the
assumption that $\chi_p(G) = 5$.
Hence, for every vertex $v \in V(G)$, we have $\chi_p(G-v) \ge 3$, and therefore the packing coloring gap satisfies $\mu_p(G) \le 2$.

\medskip
\noindent\textbf{Case 3: $\chi_p(G) = 6$.}
Suppose, to the contrary, that there exists a vertex $v \in V(G)$ such that
$\chi_p(G-v) = 3$.
By Corollary~\ref{cor:chi3-structure}, the graph $G - v$ must be a disjoint union of components, each having the structure of $S \odot K_1$, where $S$ is a star.

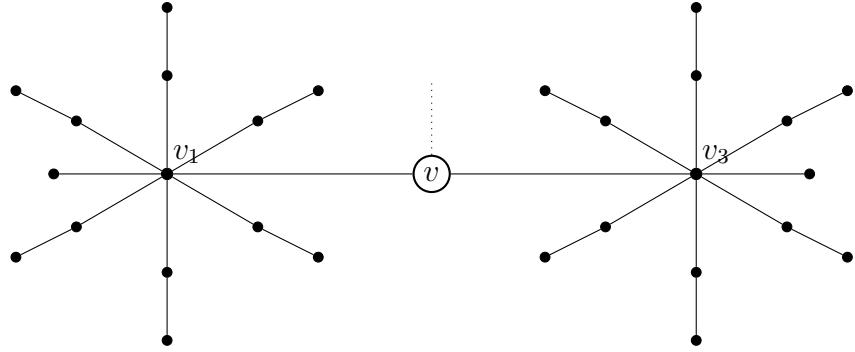
\begin{figure}[h]
\centering
\begin{tikzpicture}[
scale=1,
every node/.style={circle,draw,fill=black,inner sep=1.3pt},
vstyle/.style={circle,draw,fill=white,inner sep=2pt,thick}
]

\node[vstyle] (v) at (0,0) {$v$};
\draw[dotted] (v.north) -- ++(0,1cm);

\node[fill=black, circle, inner sep=1.5pt, label=above right:$v_1$] (cL) at (-3.5,0) {};
\node (newL) at (-5,0) {};
\node (a1L) at (-3.5,1.3) {};
\node (a2L) at (-2.3,0.7) {};
\node (a3L) at (-2.3,-0.7) {};
\node (a4L) at (-3.5,-1.3) {};
\node (a5L) at (-4.7,-0.7) {};
\node (a6L) at (-4.7,0.7) {};

\foreach \x in {newL,a1L,a2L,a3L,a4L,a5L,a6L}
  \draw (cL)--(\x);


\node (b1L) at (-3.5,2.2) {};
\node (b2L) at (-1.5,1.1) {};
\node (b3L) at (-1.5,-1.1) {};
\node (b4L) at (-3.5,-2.2) {};
\node (b5L) at (-5.5,-1.1) {};
\node (b6L) at (-5.5,1.1) {};

\foreach \i/\j in {a1L/b1L,a2L/b2L,a3L/b3L,a4L/b4L,a5L/b5L,a6L/b6L}
  \draw (\i)--(\j);

\draw (v)--(cL);

\node[fill=black, circle, inner sep=1.5pt, label=above right:$v_3$] (cR) at (3.5,0) {};
\node (new) at (5,0) {};
\node (a1R) at (3.5,1.3) {};
\node (a2R) at (4.7,0.7) {};
\node (a3R) at (4.7,-0.7) {};
\node (a4R) at (3.5,-1.3) {};
\node (a5R) at (2.3,-0.7) {};
\node (a6R) at (2.3,0.7) {};

\foreach \x in {new,a1R,a2R,a3R,a4R,a5R,a6R}
  \draw (cR)--(\x);

\node (b1R) at (3.5,2.2) {};
\node (b2R) at (5.5,1.1) {};
\node (b3R) at (5.5,-1.1) {};
\node (b4R) at (3.5,-2.2) {};
\node (b5R) at (1.5,-1.1) {};
\node (b6R) at (1.5,1.1) {};

\foreach \i/\j in {a1R/b1R,a2R/b2R,a3R/b3R,a4R/b4R,a5R/b5R,a6R/b6R}
  \draw (\i)--(\j);

\draw (v)--(cR);

\end{tikzpicture}
\caption{The graph obtained by reinserting the vertex $v$ to the components of $G-v$. The dotted line incident to $v$ indicates that $v$ may have additional leaves attached to it.}
\label{fig:extended-structure}
\end{figure}

By reinserting the vertex $v$, we either obtain the configuration illustrated in Figure~\ref{fig:extended-structure}, in which the central path has order three, namely $v_{1} v v_{3}$, and its vertices can be colored with $3,4,$ and $2$, respectively; or there exists a vertex together with its leaf between $v$ and the left component, and similarly between $v$ and the right component. In this case, the central path has order at most five and can be colored consecutively with the pattern $3,4,2,5,3$.

In both situations, the corresponding $L_{1}$-vertices are colored with $1$, while the $L_{2}$-vertices are colored according to the strategy defined in Proposition~\ref{prop:corona-caterpillar}: all vertices in $L_2$ receive color $2$, except those belonging to $L_2(v_i)$, where $v_i$ is colored with $2$ on the central path; in this case, the vertices of $L_2(v_i)$ are colored with $3$.
Hence, we obtain a valid packing coloring of $G$ using at most five colors, contradicting the assumption that $\chi_p(G)=6$.

Therefore, for every vertex $v \in V(G)$, we have $\chi_p(G-v)\ge 4$, and consequently $\mu_p(G)\le 2$.

\medskip
\noindent\textbf{Case 4: $\chi_p(G) = 7$.}
Suppose, for contradiction, that there exists a vertex $v \in V(G)$ such that $\chi_p(G-v) = 4.$\\
We reinsert the vertex $v$ and examine the conflicts that may occur when trying
to extend a packing coloring of $G-v$ to $G$. There are two types of conflicts:

\medskip
\noindent\emph{Type 1: Two vertices conflict with each other.}  
In this case, we can recolor one of the conflicting vertices with a new color $5$,
assign $v$ the new color $6$, and color the $L_1(v)$ and $L_2(v)$-vertices with $1$ and $2$ respectively. This produces a valid coloring of $G$ using only $6$ colors, contradicting $\chi_p(G) = 7$.

\medskip
\noindent\emph{Type 2: Four vertices conflict with each other.}  
The sequences of colors along the central path before and after $v$ can take one
of the following forms:
\[
4\,2\,v\,2\,4,\quad 4\,3\,v\,3\,4,\quad 3\,4\,v\,3\,4,\quad 43v43
\]
In what follows, we consider each conflict separately. We recolor $G$ under the restriction that colors $2$ and $3$ are not used consecutively on the central path. 

\smallskip
\noindent\textbf{Conflict 2.1: Sequence $4\,2\,v\,2\,4$.}  
\begin{itemize}
    \item If the sequence is $1\,4\,2\,v\,2\,4\,1$, then the vertex $v_{i-4}$ can only receive color 2 or 3:
    \begin{itemize}
        \item If $v_{i-4}$ is colored by $2$, then recolor the vertices from $v_{i-3}$ to $v_{i+3}$ by the sequence $1\,5\,3\,4\,2\,6\,1$.
         \item If $v_{i-4}$ is colored by $3$, then recolor the vertices from $v_{i-3}$ to $v_{i+3}$ by the sequence $1\,2\,5\,4\,2\,6\,1$.
    \end{itemize}

    \item If the sequence is $3\,4\,2\,v\,2\,4\,3$, then the vertices $v_{i-4}$ and $v_{i-5}$ can only receive the colors $1$ and $2$.  
    \begin{itemize}
        \item If $v_{i-5}v_{i-4}$ are colored $2\,1$, the $L_1$-vertices of $v_{i-4}$ cannot
      be colored from the set $\{1,2,3,4\}$, contradiction with the fact that $G-v$ is colored by the color set $\{1,2,3,4\}$ . 
      \item  If $v_{i-5}v_{i-4}$ are colored $1\,2$, the $L_1$-vertices of $v_{i-5}$ cannot
      be colored from the set $\{1,2,3,4\}$, a contradiction.
    \end{itemize} 
   Consequently, there exists at most one vertex on the central path before $v_{i-3}$ and
similarly at most one vertex after $v_{i+3}$. Hence, the central path contains at
most 9 vertices and can be colored using the sequence
$2\,4\,3\,5\,2\,6\,3\,4\,2$.
    \item If the sequence is $3\,4\,2\,v\,2\,4\,1$, then the vertices $v_{i-4}$ and $v_{i-5}$ can only receive colors $1$ and $2$:
    \begin{itemize}
        \item If $v_{i-4}$ receives the color $1$, then we recolor the vertices from $v_{i-3}$ to $v_{i+3}$ by the sequence $3\,2\,4\,5\,2\,6\,1$.
        \item If $v_{i-4}$ receives the color $2$, and $v_{i-5}$ receives the color $1$, then using the same arguments as in the previous case, we deduce that there exists at most one vertex before $v_{i-3}$ and it receives the color $2$, so we recolor the vertices from $v_{i-3}$ to $v_{i+3}$ by the sequence $4\,3\,5\,6\,2\,4\,1$.
    \end{itemize}
\end{itemize}

\smallskip
\noindent\textbf{Conflict 2.2: Sequence $4\,3\,v\,3\,4$.}  
\begin{itemize}
    \item If the sequence is colored by $2\,4\,3\,v\,3\,4\,2$, then recolor the vertices from $v_{i-3}$ to $v_{i+3}$ by $5\,2\,4\,3\,6\,2\,5$.
    \item If the sequence is colored by $2\,4\,3\,v\,3\,4\,1$, then recolor the vertices from $v_{i-3}$ to $v_{i+3}$ by $2\,4\,3\,6\,2\,5\,1$.
\end{itemize}

\smallskip
\noindent\textbf{Conflict 2.3: Sequence $3\,4\,v\,3\,4$.}\\  
Recolor the vertices from $v_{i-2}$ to $v_{i+2}$ by $3\,4\,2\,5\,6$.

\smallskip
\noindent\textbf{Conflict 2.4: Sequence $4\,3\,v\,4\,3$.}\\ 
Recolor the vertices from $v_{i-2}$ to $v_{i+2}$ by  $4\,5\,2\,6\,3$.

\medskip
In all cases, after recoloring the central path, we color the vertices in $L_1$ and $L_2$ following the strategy introduced in Proposition \ref{prop:corona-caterpillar} in which $L_2$ vertices are colored with color $2$, except for those belonging to
$L_2(v_i)$ where $v_i$ receives color $2$ on the central path; in that case, the vertices of $L_2(v_i)$ are colored with color $3$. After treating the conflicts, we obtain a valid coloring of $G$ using at most 6 colors,
contradicting the assumption that $\chi_p(G) = 7$. Therefore, no vertex $v \in V(G)$
can produce a decrease of 3 in the packing chromatic number, $\chi_p(G-v)$, and we conclude that
\[
\mu_p(G) \le 2.
\]

This completes the proof of Theorem~\ref{thm:corona-gap}.

\end{proof}
\section{Unbounded Packing Coloring Gap} \label{unboundedSection}

Cographs form a very wide class of graphs as they can be constructed
recursively from single-vertex graphs using disjoint union and join
operations. While this class may include graphs with small or bounded
packing coloring gaps, the recursive use of join operations can also
produce graphs where vertex deletions cause a large reduction in the
packing chromatic number. Indeed, we can construct explicit examples
of cographs with arbitrarily large packing coloring gaps.

\medskip
\noindent\textbf{Example.}  
Figure~\ref{fig:cograph-large-gap} illustrates a specific cograph constructed by performing a join operation of the disjoint union of $t$ edges with a single vertex $v$. The resulting graph consists of $t$ copies of $K_3$, each sharing the common central vertex $v$, as shown in the figure. This graph is known as the \emph{friendship graph}.

In this construction, we precisely determine the packing chromatic number and show that $\chi_p(G)=t+2$. Indeed, one may assign color $2$ to the central vertex $v$, and color $1$ to one vertex in each triangle. The remaining vertices that consist of one vertex of each triangle require $t$ additional colors, so $\chi_\rho(G)\leq t+2$. Now suppose that $G$ admits a packing coloring with $t+1$ colors. First, observe that the color $1$ can be assigned to at most $t$ vertices; indeed, in each triangle, at most one vertex can receive the color $1$, since any two vertices in the same triangle are at a distance at most $2$.
Thus, at least $t+1$ vertices must be colored with colors from $\{2,3,\dots,t+1\}$, which is a set of $t$ colors. So, there exist two vertices that receive the same color from this set. However, since the graph has diameter $2$ , then no two vertices can receive the same color $i$ for any $i \ge 2$.
Therefore, $\chi_p(G) \ge t+2$, and therefore $\chi_p(G) = t+2$.

On the other hand, removing the central vertex $v$ yields a disjoint union of edges, whose packing chromatic number is $2$. Therefore, the packing coloring gap is
\[
\chi_\rho(G) - \chi_\rho(G - v) = (t+2) - 2 = t.
\]
Consequently, the packing coloring gap can be arbitrarily large by increasing the number of triangles in the construction.
  
\begin{figure}[h]
\centering
\begin{tikzpicture}[scale=1.2, every node/.style={circle,draw,fill=black,inner sep=1.5pt}]
    \node (v0) at (0,0) {};

    \node (v1) at (2,0) {};
    \node (v2) at (1,1.5) {};
    \draw (v0) -- (v1) -- (v2) -- (v0);
    \node[draw=none,fill=none] at (1,0.6) {$T_2$};

    \node (v3) at (-2,0) {};
    \node (v4) at (-1,1.5) {};
    \draw (v0) -- (v3) -- (v4) -- (v0);
    \node[draw=none,fill=none] at (-1,0.6) {$T_1$};

    \node (v5) at (-0.5,-1.8) {};
    \node (v6) at (-1.7,-0.9) {};
    \draw (v0) -- (v5) -- (v6) -- (v0);
    \node[draw=none,fill=none] at (-0.7,-0.8) {$T_t$};
\draw[dash pattern=on 0pt off 3pt, line cap=round, line width=1.2pt] (v1) to[bend left] (v5);
    \node[draw=none,fill=none] at (0.3,-0.3) {$v$};
\end{tikzpicture}
\caption{A family of cographs with unbounded packing coloring gap.}

\label{fig:cograph-large-gap}
\end{figure}
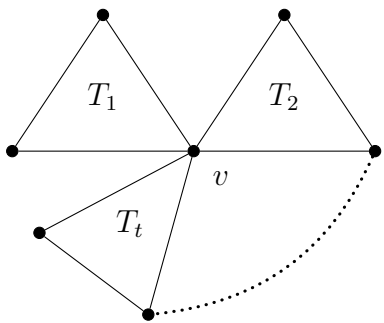

This demonstrates that, although some cographs may have bounded gaps, in
general the packing coloring gap for cographs can be arbitrarily large.
Such examples motivate the restriction of our study to graphs with
small or bounded gaps, where structural analysis is feasible.

\medskip

This construction also provides other examples with unbounded packing coloring gap. A natural generalization of the friendship graph is given by windmill graph, which is obtained by taking $t$ copies of a complete graph $K_m$ and identifying one vertex in common. Indeed, let $G$ be a windmill graph. Arguing as above, one can assign color $2$ to the central vertex and color $1$ to one vertex in each copy of $K_m$. The remaining $t(m-2)$ vertices must receive distinct colors. Therefore, $\chi_\rho(G) = t(m-2) + 2$.

On the other hand, removing the central vertex yields a disjoint union of $t$ copies of $K_{m-1}$, whose packing chromatic number is $m-1$. Hence, the packing coloring gap is $\chi_\rho(G) - \chi_\rho(G - v) = \bigl(t(m-2) + 2\bigr) - (m-1) = t(m-2) - m + 3$ which can be made arbitrarily large as $t$ increases.
\section{Open Problems}
\label{sec:open}

The results obtained in this paper naturally motivate further questions
concerning the behavior of the packing coloring gap on graph classes whose
packing chromatic number is already well understood.

A particularly relevant class is that of \emph{lobsters}, that is, a tree in which the removal of all leaves and their incident edges results in a caterpillar. 
Equivalently, a lobster is a tree in which every vertex is at distance at most $2$ from a central path.
The packing chromatic number of lobsters was studied by Argiroffo, Nasini,
and Torres~\cite{Argiroffo2014}, who introduced a structural parameter $c_T$
defined as follows.
Let $T$ be a lobster with central path $P$.
For each vertex $v\in V(P)$, let $N_4(v)$ denote the set of neighbors of $v$
outside the central path that have degree at least four.
The parameter $c_T$ is then defined by
\[
c_T=\max_{v\in V(P)} |N_4(v)|.
\]
In~\cite{Argiroffo2014}, it was shown (Theorem~3.8) that if $c_T\le 1$, then the packing chromatic number of $T$ is bounded by a constant.
It is therefore natural
to ask whether a similar restriction also yields a small packing coloring gap.
In particular, one might expect that the condition $c_T\le 1$ forces
$\mu_p(T)$ to remain bounded by~$2$.
However, explicit examples indicate that this is not always
the case. Indeed, we found several lobsters $T$ that satisfy $c_T\le 1$ for
which $\mu_p(T)=3$.

This leads to the following open problem:

\begin{problem}
Is it true that $\mu_p(T)\le 3$ for every such
lobster?
\end{problem}

\begin{problem}
Determine how the parameter $c_T$ influences the packing coloring gap for larger values of $c_T$. In particular, establish whether additional structural restrictions on lobsters, or more generally on classes of trees, lead to sharper bounds on $\mu_p(G)$.
\end{problem}

In all the cases where we established a bounded $\mu_p(G)$, it was also the case that $\chi_p(G)$ is bounded. This naturally leads to the following open problem.

\medskip
\begin{problem}
Do there exist graph classes for which $\mu_p(G)$ remains bounded while $\chi_p(G)$ is unbounded? More generally, how can such classes be characterized or identified?
\end{problem}


\end{document}